\documentclass[12pt]{amsart}
\usepackage[T1]{fontenc}
\usepackage{lmodern}
\usepackage{amsmath,amssymb,amsthm}
\usepackage[margin=1in]{geometry}
\usepackage{microtype}
\usepackage{hyperref}
\hypersetup{hidelinks}
\newcommand{\SG}{\mathcal{SG}}
\newcommand{\crk}{\operatorname{cr}}
\newcommand{\CE}{\operatorname{CE}}
\newcommand{\tarrow}{\Longrightarrow}
\newcommand{\Z}{\mathbb Z}

\newcommand{\Nzero}{\mathbb N_0}
\newtheorem{theorem}{Theorem}[section]
\newtheorem{lemma}[theorem]{Lemma}
\newtheorem{proposition}[theorem]{Proposition}
\newtheorem{corollary}[theorem]{Corollary}
\theoremstyle{definition}
\newtheorem{definition}[theorem]{Definition}
\theoremstyle{remark}
\newtheorem{remark}[theorem]{Remark}
\numberwithin{equation}{section}

\usepackage{tikz}
\usepackage{tikz-cd}
\usetikzlibrary{decorations.pathreplacing,arrows.meta,calc}

\begin{document}
\title{4-manifold topology and collision rank of groups}
\author{Vyacheslav Krushkal}
\address{Department of Mathematics, University of Virginia, Charlottesville, VA 22904}
\email{krushkal\char 64 virginia.edu}
\begin{abstract}
We introduce collision rank and show that groups of finite collision 
rank are good in the sense of topological $4$-manifold theory. Finitely generated groups of collision rank two are shown to exist outside the
class of subexponentially amenable groups. Such examples are given by commutator
subgroups of topological full groups of Sturmian systems, and include continuum many pairwise nonisomorphic infinite simple amenable groups.
\end{abstract}
\maketitle

\section{Introduction}

The disk embedding theorem is a key result underlying the theory of topological $4$-manifolds. It is used in the proofs of the $4$-dimensional surgery exact sequence and of the $5$-dimensional $s$-cobordism theorem in the topological category. Freedman proved the disk embedding theorem in the simply connected case in his groundbreaking work \cite{F0}. Groups for which the disk embedding theorem holds are called good. It was shown in \cite{F1} that elementary amenable groups are good. The class of known good groups was extended to include subexponentially amenable groups, denoted ${\SG}$, in \cite{FT,KQ}. Here $\SG$ is the smallest class of groups that contains all finitely generated groups of subexponential growth and is closed under extensions, direct limits, subgroups, and quotients. It was conjectured in \cite{F1} that nonabelian free groups are not good; this conjecture remains open, see \cite{FK} for a detailed discussion.

The main result of this paper is a construction of good groups outside the class $\SG$. It follows from the next two theorems.

\begin{theorem}\label{thm:main}
Groups of finite collision rank are good.
\end{theorem}

Collision rank, denoted $\crk(G)$, is defined in
Section~\ref{sec: collision}. Roughly speaking, it measures the number of successive
collision steps needed to reduce finite sets of group elements to the
identity. Groups of collision rank at most one are precisely the supramenable groups.

\begin{theorem}\label{thm:main-intro}
Let $\alpha=[0;a_1,a_2,\ldots]$ be an irrational number with unbounded continued fraction coefficients,
let $(X_\alpha,T_{\alpha})$ be the
Sturmian subshift of slope $\alpha$, and let
$G_\alpha=[[T_\alpha]]'$ be the commutator subgroup of its topological
full group. Then $G_\alpha$ is finitely generated, infinite, simple, and
amenable. Moreover,
$ G_\alpha\notin\SG$ and $\crk(G_\alpha)=2.$
\end{theorem}

This theorem gives continuum many pairwise nonisomorphic groups; 
for example, it applies to $\alpha=[0;2,3,4,\ldots]$ 
In particular, this gives the first instances of finitely generated simple good groups of exponential growth.

Sturmian subshifts arise from encoding irrational rotations of the circle;
modern theory of their underlying symbolic sequences in symbolic dynamics was initiated by Morse and Hedlund in \cite{MH}.
The study of full groups originates in Dye’s work on orbit equivalence in measurable dynamics \cite{Dye}. 
Topological full groups have long been studied in Cantor dynamics \cite{GPS}, and their commutator subgroups, provide examples of finitely generated infinite simple amenable groups.

The new result in Theorem \ref{thm:main-intro} is the computation of
collision rank. Finite generation and simplicity were proved by Matui
\cite{Matui}, amenability by Juschenko--Monod \cite{JM}, and exclusion
from $\SG$ by Grigorchuk--Medynets \cite{GM}. 

The proof of Theorem \ref{thm:main-intro} consists of two main parts. Long periodic blocks give a bound on the number
of cosets of a half-orbit stabilizer represented by products of a suitable
length. This bound is smaller than the number of leaves of the corresponding collision tree, so two leaf products then have
a difference group element in this stabilizer. The stabilizer is locally finite \cite{GM}, so the
finitely many possible differences generate a finite group, and a second
collision gives the identity. 

The topological argument---the proof of Theorem \ref{thm:main}---may be thought of as an iteration of the argument of
\cite{KQ} proving that subexponential groups are good. Roughly speaking, starting with a finite collection of group elements representing double point loops of a capped grope, each collision step is realized by geometric moves on the capped grope. The result is a collection of double point loop elements in the next prescribed finite set. Repeating
this on a grope of sufficient height gives a $\pi_1$-null capped grope, sufficient for the proof of the disk embedding theorem.

The new good groups, constructed in Theorem \ref{thm:main-intro}, are not finitely presented \cite[Theorem 5.7]{Matui}, so they cannot occur as fundamental groups of compact manifolds.
It is a natural question
whether these methods may produce finitely presented good groups outside $\SG$.
In Proposition \ref{thm:free} we show that nonabelian free groups have infinite  collision rank, so the criterion of Theorem~\ref{thm:main}
does not settle whether they are good.

The paper is organized as follows. Section \ref{sec: collision} introduces the collision rank, giving two equivalent definitions: one using trees and a game-theoretic one. The proof of Theorem \ref{thm:main} is given in Section \ref{sec:topology}. Sections \ref{sec:dynamics} and \ref{sec:examples} prove Theorem \ref{thm:main-intro}, starting from a brief discussion of topological full groups, and concluding with the construction of an uncountable collection of Sturmian examples. Section \ref{sec:free} discusses properties of the collision rank and gives a proof that for nonabelian free groups it is infinite.

{\bf Acknowledgements.} I am grateful to Mike Freedman for numerous conversations on all aspects of $4$-manifold topology over the years, and for the suggestion to use AI to study it. 

The main topological input, motivating the collision rank definition, is based on the methods developed in \cite{KQ}. I would like to thank Frank Quinn for sharing his insight and for many discussions of these techniques. 

This work was supported in part by NSF grant DMS-2405044, Renaissance Philanthropy’s AI for Math Seed Grant, and by Logical Intelligence.

{\bf AI disclosure.} The author formulated the original collision property in groups and its topological applications, including the proof of Theorem \ref{thm:main}. 
Substantial parts of the group-theoretic strategy were developed during an extended dialogue with ChatGPT-5.6 Sol. In particular, it was instrumental in identifying derived topological full groups of Sturmian systems as candidates and in developing the half-orbit stabilizer coset counting and periodic block arguments used to prove finite collision rank for these groups in Sections \ref{sec:dynamics}-\ref{sec:examples}. ChatGPT-5.6 Sol and Claude Opus 5 were used also to develop and check properties of the collision rank (Section \ref{sec:free}) and to draw Figure \ref{fig:periodic}. The author checked, revised, and takes responsibility for all arguments.

\section{collision rank}\label{sec: collision}

Given a group $G$, throughout the paper a {\em label set} will refer to a finite symmetric subset
containing $1$. 

Let $T_h$ be the complete rooted trivalent tree of height $h$. Every leaf of $T_h$ is
at distance $h$ from the root, and there are a total of $3\cdot2^{h-1}$ leaves.
Two vertices are called {\em incomparable} if neither one lies on the geodesic path from the
root to the other. 

\begin{definition}\label{def:stage}
Given two label sets $A,B\subseteq G$, we write $A\tarrow B$ if there is an
integer $h\geq1$ such that for  every edge labeling of $T_h$ by elements of
$A$, there exist two incomparable vertices $v,w$ with
\[
 p(v)\cdot p(w)^{-1}\in B.
\]
Here $p(v)=a_k\cdots a_1\in G$, where $a_1,\ldots,a_k$ are the ordered edge labels along
the path from the root to $v$.
\end{definition}

Since $B$ is symmetric, the order
of $v,w$ is not important. Denote $A^h:=\{a_h\cdots a_1:a_i\in A\}$.

\begin{definition}\label{def:rank}
The {\em collision rank} $\crk(G)$ is the least
$r\in\Nzero$ such that every label set $A_0\subseteq G$ admits
label sets $A_1,\ldots,A_r$ with
\[
 A_0\tarrow A_1\tarrow\cdots\tarrow A_r=\{1\}.
\]
If no such $r$ exists, set $\crk(G)=\infty$.
\end{definition}

By definition, $\crk(G)=0$ if and only if $G$ is trivial.
The definition for $r=1$ is motivated by the group-theoretic effect of the geometric moves in the proof in \cite{KQ} that groups of subexponential growth are good, discussed in the next section. Moreover, the condition $\crk(G)\leq 1$ characterized supramenability, see Proposition \ref{lem:supramenability}.

\begin{remark}
There is an equivalent game-theoretic definition of collision rank. There are two players, Pruner and Labeler. Starting with a single root, Pruner adds three edges at the first move and thereafter adds two edges at a chosen leaf. After each move, Labeler assigns an element of $A$ to each new edge. Pruner wins after finding incomparable
vertices whose difference group element $p(v)\cdot p(w)^{-1}$ lies in $B$. It is not difficult to see that $A\tarrow B$ if and only if Pruner has a winning strategy.
\end{remark}

We discuss properties of the collision rank in more detail in Section \ref{sec:free}. 
The following observation isolates the group-theoretic argument used in the proof of Theorem \ref{thm:main-intro}. Note that the subgroup $H$ is not assumed to be normal.

\begin{lemma}\label{lem:coset-criterion}
Let $A$ be a label set in $G$. Let $H\leq G$ be a locally
finite subgroup, and suppose there exists $h\geq1$ such that
\begin{equation}\label{eq:coset criterion}
 \#\{Hg:g\in A^h\}<3\cdot2^{h-1}.
\end{equation}
Then, setting $C=H\cap A^h(A^h)^{-1}$,
\[
 A\tarrow C\tarrow\{1\}.
\]
Therefore if such $H$ and $h$ exist for every label set $A$, then $\crk(G)\leq2$.
\end{lemma}

\begin{proof}
The set $C$ is finite, symmetric, and contains $1$. By
\eqref{eq:coset criterion}, two leaves of any $A$-labeling of $T_h$
have products $f,g$ in the same coset $Hf=Hg$. Thus $fg^{-1}\in C$,
which proves $A\tarrow C$. By assumption, the subgroup
$\langle C\rangle$ is finite. A second tree with more than
$|\langle C\rangle|$ leaves has two equal leaf products under every
$C$-labeling, proving $C\tarrow\{1\}$.
\end{proof}

\begin{lemma}\label{lem:free-monoid}
If $G$ contains a free submonoid of rank three, then $\crk(G)\geq2$.
\end{lemma}

\begin{proof}
Complete three free monoid generators to a label set $A$. For every
$h$, label the root edges of $T_h$ by the three generators, and the
two outgoing edges at every other internal vertex by distinct
generators. This labeling has no collision to
$\{1\}$, so $A\not\tarrow\{1\}$.
\end{proof}

\section{Application to 4-manifold topology}\label{sec:topology}

This section contains the proof of Theorem \ref{thm:main}. The proof is based on an iteration of the construction of \cite{KQ}. We review the ideas of that proof and explain how they apply in the present context. 
The geometric step realizes a path in a
tree with edges labeled by group elements by modifying intersections between caps and then moving them down to the base surface of a given capped grope. The group elements of
the resulting double points are the product along the path. We use this step with products in a prescribed symmetric set, thus geometrically implementing a single collision stage $A_i\tarrow A_{i+1}$. When the fundamental group $\pi$ of the ambient $4$-manifold has finite collision rank, after an iterated application of this step all double point loops represent the trivial element. Note that the term {\em collision} had a different usage in \cite[Section 2.6]{KQ}, referring to a certain type of intersection points. The collision rank in this paper is an unrelated concept.

We use the $\pi_1$-null disk formulation of good groups; see
\cite{FQ,FT,BKKPR}. As is common in the subject, we will interchangeably use the term capped grope to describe both the underlying $2$-complex and its untwisted $4$-dimensional thickening; the precise meaning will be clear from context. Given a
disk-like capped grope of height
$3/2$ and a homomorphism from its fundamental group to $\pi$, one must find
an immersed disk with the same framed attaching region whose double point
loops map trivially to $\pi$. Grope height raising allows us to start with
any prescribed finite height. To simplify counting in the inductive construction, we will work with gropes of height $2$ rather than $3/2$, and will produce a capped grope whose fundamental group has trivial image in $\pi$.
All constructions take place in the given thickening and preserve the
attaching region. We use the usual convention that the body of a properly
immersed capped grope is embedded and all double points are among caps.

Fix base arcs for the capped grope at the beginning of the argument, and record the images under the given
homomorphism to $\pi$ of double
point loops using both orders of the sheets. Thus the set of labels is
symmetric. Grope splitting, introduced in \cite{K} and illustrated in Figure \ref{fig:splitting}, carries these basings to the new
branches. As discussed below in more detail, the proof in \cite{KQ} is based on a systematic application of grope splitting, uniformizing the pattern of intersections between gropes to any given distance. 

\begin{figure}[ht]
\centering
\includegraphics[height=5cm]{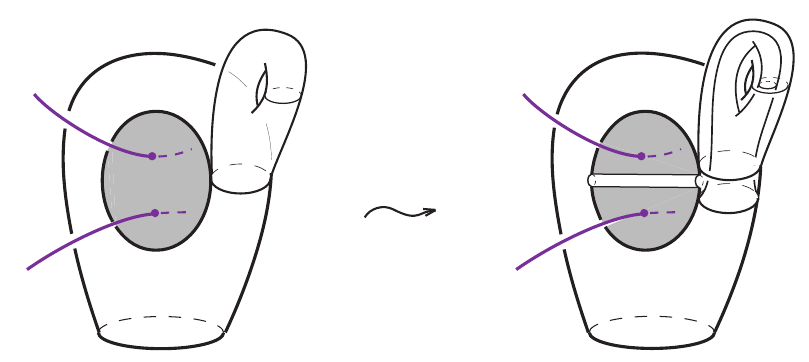}
\caption{Grope splitting. If two caps (drawn purple), intersecting a given cap, have different group elements or different dyadic labels, the cap may be split into two at the expense of increasing the genus of the base surface. The dual subgropes for the two new caps are parallel copies of the original one.}
\label{fig:splitting}
\end{figure}

The following lemma is a geometric counterpart (and the key motivation) for a group-theoretic collision stage in Definition \ref{def:stage}.

\begin{lemma}[Inductive step: one geometric stage]\label{lem:geometric-stage}
Let $A,B\subseteq\pi$ be label sets with $A\tarrow B$.
Suppose a properly immersed disk-like capped grope has height 
$k+2$, where $k\geq 0$, and all its double point group elements lie in $A$.
Then its thickening contains a properly immersed capped grope of height $k$, with the same framed attaching region, all of whose double
point group elements are in $B$.
\end{lemma}

\begin{proof}
The bottom $k$ stages will be kept fixed. Regard the top two stages 
as a collection of capped gropes of height $2$, with basing arcs
inherited through the lower stages. The construction will replace these
top gropes by immersed disks.

Let $h$ be the tree height certifying $A\tarrow B$. Recall from Definition \ref{def:stage} that this means: for any labeling of the tree $T_h$ by elements of A, there exist two incomparable vertices $v,w$ with $p(v)\cdot p(w)^{-1}\in B$.  Split the top height $2$ gropes into dyadic branches
with uniform $2h$-types and with no geometric collisions up to distance $2h$, using \cite[Sections 2.5--2.6]{KQ}. Here a geometric collision
is the intersection configuration defined in that reference, distinct
from a collision of group labels. These two properties are preserved
under further splitting.

A branch with a free cap can be contracted across that cap without
creating double points. After all such contractions, starting from a branch (a genus 1 part of the base surface) the construction
of \cite[Section~4.1]{KQ} gives an intersection tree of height $h$. Vertices correspond to branches and edges record intersections of
caps, Figure \ref{fig:grope_tree}. There are three edges at the root and at two new edges at every subsequent vertex. We will use a
trivalent tree of height $h$, given by this construction.

\begin{figure}[ht]
\centering
\begingroup
\setlength{\unitlength}{1cm}
\begin{picture}(15.176,4.5)
  \put(0,0){\includegraphics[height=4.5cm]{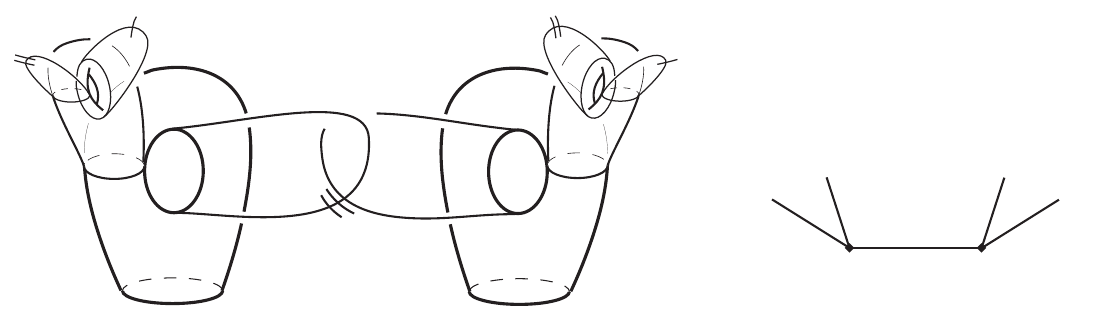}}
  \put(0.25,4.00){\makebox(0,0){\small $g_1$}}
  \put(2.15,4.23){\makebox(0,0){\small $g_2$}}
  \put(4.63,1.05){\makebox(0,0){\small $g_3$}}
  \put(7.25,4.23){\makebox(0,0){\small $g_4$}}
  \put(9.37,4.00){\makebox(0,0){\small $g_5$}}
  \put(10.94,1.23){\makebox(0,0){\small $g_1$}}
  \put(11.80,1.82){\makebox(0,0){\small $g_2$}}
  \put(12.64,0.72){\makebox(0,0){\small $g_3$}}
  \put(13.40,1.82){\makebox(0,0){\small $g_4$}}
  \put(14.49,1.23){\makebox(0,0){\small $g_5$}}
\end{picture}
\endgroup
\caption{}
\label{fig:grope_tree}
\end{figure}

To compare conventions between this paper and \cite{KQ}, let $q(v)$ be the product of group elements along the geodesic path from the root to $v$, in that order. Label the
abstract tree by their inverses, which still belong to $A$. The convention
of Definition \ref{def:stage} then gives $p(v)=q(v)^{-1}$. Thus $A\tarrow B$
provides incomparable vertices $v,w$ with
\[
 q(v)^{-1}\cdot q(w)=p(v)\cdot p(w)^{-1}\in B.
\]
The common initial portion of the two root paths cancels. This is exactly
the product along the geodesic from $v$ to $w$, whose length is at most
$2h$. Incomparability ensures that the path enters and leaves each
intermediate branch through different caps. Uniform types ensure that all cap intersections encountered along the path have matching group elements. If a required intersection has
disappeared in an earlier contraction, a free cap occurs along the path
and can instead be contracted, as in \cite[Section~4.1]{KQ}.

We apply the branch elimination of \cite[Section~4.2]{KQ} to this path.
Write its group element labels as $g_1,\ldots,g_m$ and set
$g:=g_1\cdots g_m\in B$. Contract the first branch across its
chosen cap. The new intersections between a base surface and caps have
group elements $g_1$. Push them down and back up through the next chosen
cap. Their labels become $g_1 g_2$. Continuing along the path
gives base-cap intersections with label $g$; pushing these down at the
last branch gives base-base intersections with group element $g\in B$.
With the opposite sheet order their label is $g^{-1}$, also in $B$.

The verification in \cite[Section 4.2]{KQ} that no other base-base
intersections occur uses the absence of geometric collisions through
distance $2h$. The fixed base arcs give the displayed products
literally, and not up to conjugation.

Each such operation removes branches and creates only base-base
intersections labeled in $B$. The remaining cap intersections retain
their types and group element labels in $A$. Contract newly free caps and repeat.
There are finitely many branches, so the process ends with immersed
disks replacing all the top gropes. These disks are the new caps on the
unchanged lower body. Their double point labels lie in $B$, as required.
\end{proof}

\begin{proof}[Proof of Theorem~\ref{thm:main}]
Let $r=\crk(\pi)<\infty$. If $r=0$, then $\pi$ is trivial and hence good. Therefore assume $r\geq 1$. Start with a capped grope of height
$2$ and the given homomorphism to $\pi$. Raise its height to $2r$,
and let $A_0$ be a label set containing its double point
labels. As in Definition~\ref{def:rank}, choose a chain
\[
 A_0\tarrow A_1\tarrow\cdots\tarrow A_r=\{1\}.
\]
Apply Lemma~\ref{lem:geometric-stage} successively, beginning at the top.
After $i$ steps the capped grope has height $2(r-i)$, and
its double point labels lie in $A_i$. The final application of the lemma gives a $\pi_1$-null disk.
This proves that $\pi$ is good.
\end{proof}

For the groups considered in Sections \ref{sec:dynamics}, \ref{sec:examples} the argument has a particularly direct form.
Given the initial labels $A$, we find a locally finite subgroup $H$
and a height $h$ for which there are fewer $H$-cosets represented
by $A^h$ than leaves of the collision tree. Two leaves therefore give
a difference group element in the finite symmetric set
\[
 C=H\cap A^h(A^h)^{-1}.
\]
The first geometric stage puts all double point labels in $C$.
Since $H$ is locally finite, the subgroup generate by $C$ is finite (of size depending in general on the initial label set $A$), and a second geometric
stage makes every group element trivial. This iteration is analogous to the argument used for
an extension of subexponential groups; here coset counting enables the
first collision stage in place of subexponential growth in a quotient. In our case the groups $\pi$ are simple and $H$ is not a normal subgroup.

\section{Subshifts and topological full groups}\label{sec:dynamics}
The following two sections give a proof of Theorem \ref{thm:main-intro}. 
We refer the reader to \cite[Sections 2-3]{GM} for an introductory discussion of algebraic properties of topological full groups, and to \cite{LM} for general background on symbolic dynamics, and in particular Sturmian systems.

Let $X\subseteq\mathcal A^{\mathbb Z}$ be a minimal aperiodic subshift over
 a finite alphabet, and let $T$ be the left shift, acting in coordinates as $(Tx)_i=x_{i+1}$ for $x\in X$.  Let $\mathcal L(X)$ denote its language.  The topological full group \cite{GPS} is defined as
\[
 [[T]]=\{g\in\operatorname{Homeo}(X):g(z)=T^{c_g(z)}z,
          \ c_g\colon X\longrightarrow\mathbb Z\text{ continuous}\}.
\]
Aperiodicity makes $c_g$ unique.  The proof in Section \ref{sec:periodic-window-estimate} uses long powers in
$\mathcal L(X)$ to give a bound on the number of cosets of a locally finite subgroup, and this bound is used in Section \ref{sec:two stages} to deduce two collision stages.

\subsection{Half-line stabilizers}\label{sec:minimal-subshifts}
We note that orbit permutations and locally finite half-orbit stabilizers, used in the arguments below, also
play an important role in the amenability proof of
\cite[Sections 3-4]{JM}.
For $x\in X$, following \cite{P, GPS} define the half-line and its stabilizer:
\[
 P_x=\{T^n x:n\geq0\},\qquad
 H_x=\{g\in[[T]]:g(P_x)=P_x\}.
\]
The subgroup $H_x$ is locally finite, see
\cite[Section 5]{P}, \cite{GPS}, and also the summary in \cite[Theorem 2.2(2)]{GM}.
The orbit of $x$ will be identified with $\mathbb Z$ by $T^n x\leftrightarrow n$, this map gives orbit coordinates.
In these coordinates, $g\in[[T]]$ acts by the permutation \cite[Section 4]{JM}
\begin{equation}\label{eq:orbit-permutation}
 \pi_g^x(n)=n+c_g(T^n x).
\end{equation}
Since $c_g$ has finite image, this permutation of the integers has bounded displacement. (The group of such bijections is denoted $W({\mathbb Z})$ in \cite{JM}). 
Let $\mathbb N_0=\{0,1,2,\ldots\}$.

\begin{lemma}\label{lem:defect-dictionary}
For $f,g\in[[T]]$,
\begin{equation} \label{eq:half line} 
 H_xf=H_xg
 \quad\Longleftrightarrow\quad
 (\pi_f^x)^{-1}(\mathbb N_0)=(\pi_g^x)^{-1}(\mathbb N_0).
\end{equation}
\end{lemma}

\begin{proof}
The equality $H_xf=H_xg$ is equivalent to $fg^{-1}\in H_x$ and to
$f^{-1}P_x=g^{-1}P_x$. In orbit coordinates, the action of $f^{-1}$ on $P_x$ corresponds to the action of $(\pi_f^x)^{-1}$ on ${\mathbb N}_0$, giving the right hand side of equation (\ref{eq:half line}).
\end{proof}

\subsection{A bound on cosets in a periodic block}\label{sec:periodic-window-estimate}
In the discussion below, a {\em periodic block}, or more precisely a block of period $p$, refers to an interval $I$ on which $x$ is $p$-periodic, i.e. $x_{i+p}=x_i$
whenever $i,i+p\in I$.

Let $A\subseteq[[T]]$ be a finite nonempty subset.  By continuity of $c_g\colon X\longrightarrow\mathbb Z$ on $X$ and compactness of $X$, there exist integers $D\geq 1$ and
$\rho\geq 0$ such that, for every $a\in A$,
\[
 |c_a(z)|\leq D\quad \text{ for all } z\in X,\qquad
 c_a(z)\text{ is determined by }z[-\rho,\rho].
\]
Here $z[i,j]$ denotes
 the subword $z_i z_{i+1}\cdots z_j$ of $z$, and {\em determined by $z[-\rho,\rho]$} means that if $z[-\rho,\rho]=z'[-\rho,\rho]$ then $c_a(z)=c_a(z')$. The following proposition establishes a key bound used in the proof of Theorem \ref{thm:main-intro}, and Figure \ref{fig:periodic} illustrates the notation and the dynamics underlying the arguments.

\begin{proposition}\label{prop:periodic-window}
Let $h,p\geq1$ be integers and $D,\rho$ be defined as above. Set $R:=Dh$.  Suppose that $x \in X$ is
$p$-periodic on an interval containing
\[
 [-2R-\rho,\,2R+\rho].
\]
Then
\[
 \#\{H_xg:g\in A^h\}\leq\left(\frac{2Dh}{p}+2\right)^p.
\]
\end{proposition}

\begin{proof}
For $g=b_h\cdots b_1\in A^h$, each factor moves orbit
coordinates by at most $D$, so for all $n\in\mathbb Z$,
\begin{equation}\label{eq:bounded-displacement}
 |\pi_g^x(n)-n|\leq R.
\end{equation}
The following lemma applies the argument of
\cite[Lemma 3.5]{MB} to a block of period $p$.
\begin{lemma} \label{lem:congruence}
In the context of Proposition \ref{prop:periodic-window}, \begin{equation}\label{eq:local-periodicity}
 m,n\in[-R,R],\quad m\equiv n\pmod p
 \quad\Longrightarrow\quad
 \pi_g^x(m)-m=\pi_g^x(n)-n.
\end{equation}
\end{lemma}
{\em Proof of Lemma \ref{lem:congruence}.} Evaluate $b_1,\ldots,b_h$ from the starting points $m,n$ simultaneously.
Both trajectories stay in $[-2R,2R]$.  Whenever their difference is
$m-n$, by periodicity the radius-$\rho$ intervals around them in $x$ agree.  Since the action is determined by these radius-$\rho$ intervals, next factor then moves them
by the same amount.  Induction preserves their difference through all
$h$ factors, proving \eqref{eq:local-periodicity}. \qed

\begin{figure}[htb]
\centering
\begin{tikzpicture}[
  x=0.92cm,y=0.92cm,
  font=\small,
  axis/.style={draw=black!60,line width=0.4pt},
  move/.style={-{Stealth[length=1.8mm]},draw=black!65,line width=0.45pt},
  every node/.style={inner sep=2pt}
]
  \node[anchor=west,font=\small] at (-0.25,1.2)
    {(a) The periodic interval};
  \fill[black!5] (0,0) rectangle (12,0.55);
  \draw[axis] (0,0) rectangle (12,0.55);
  \foreach \j in {1,...,11} \draw[black!25,line width=0.3pt] (\j,0)--(\j,0.55);
  \foreach \j in {0,...,11} \node at ({\j+0.5},0.275) {$u$};
  \node[anchor=west] at (12.25,0.275) {$u^e$ in $x$};

  \filldraw[fill=black!7,draw=black!45,line width=0.4pt]
    (1,-1.25) rectangle (11,-0.7);
  \node at (6,-0.975) {$[-2R-\rho,\,2R+\rho]$};
  \node[anchor=west] at (12.25,-0.975) {read};

  \filldraw[fill=black!10,draw=black!45,line width=0.4pt]
    (2,-2.25) rectangle (10,-1.7);
  \node at (6,-1.975) {$[-2R,\,2R]$};
  \node[anchor=west] at (12.25,-1.975) {visited};

  \filldraw[fill=black!14,draw=black!45,line width=0.4pt]
    (4,-3.25) rectangle (8,-2.7);
  \node at (6,-2.975) {$[-R,\,R]$};
  \node[anchor=west] at (12.25,-2.975) {start};
  \draw[axis] (6,-3.25)--(6,-3.43);
  \node[anchor=north] at (6,-3.43) {$0$};

  \node[anchor=west,font=\small] at (-0.25,-4.55)
    {(b) A translation gives a terminal segment in each $S_r$};
  \node[anchor=east] at (1.15,-5.6) {$S_r$};
  \draw[axis] (1.7,-5.6)--(8.7,-5.6);
  \fill[black!9] (6,-7.87) rectangle (10.5,-7.09);
  \foreach \xx/\lab in {2/s_1,3.6/s_2,5.2/s_3,6.8/s_4,8.4/s_5} {
    \node[anchor=south] at (\xx,-5.4) {$\lab$};
    \draw[move] ({\xx+0.06},-5.76)--({\xx+0.94},-7.34);
  }
  \foreach \xx in {2,3.6}
    \draw[fill=white,line width=0.6pt] (\xx,-5.6) circle[radius=2.5pt];
  \foreach \xx in {5.2,6.8,8.4}
    \fill (\xx,-5.6) circle[radius=2.5pt];
  \node[anchor=west] at (10.4,-5.6) {$k_r=3$};
  \node[anchor=west] at (10.4,-6.5) {$n\mapsto n+d_r$};

  \node[anchor=east] at (1.15,-7.5) {$\pi_g^x(S_r)$};
  \draw[axis,-{Stealth[length=1.7mm]}] (1.7,-7.5)--(10.6,-7.5);
  \draw[dashed,black!65,line width=0.4pt] (6,-7.0)--(6,-7.93);
  \foreach \xx in {3,4.6}
    \draw[fill=white,line width=0.6pt] (\xx,-7.5) circle[radius=2.5pt];
  \foreach \xx in {6.2,7.8,9.4}
    \fill (\xx,-7.5) circle[radius=2.5pt];
  \node[anchor=north] at (6,-7.94) {$0$};
  \node[anchor=north] at (3.5,-7.94) {negative};
  \node[anchor=north] at (8.4,-7.94) {nonnegative};

  \node at (6.6,-8.85)
    {$S_r=\{s_1<\cdots<s_5\},\qquad
      (\pi_g^x)^{-1}(\mathbb N_0)\cap S_r=\{s_3,s_4,s_5\}.$};
  \node at (6.6,-9.6)
    {$\displaystyle (k_0,\ldots,k_{p-1})
       \ \text{determines }(\pi_g^x)^{-1}(\mathbb N_0)
       \ \text{and hence }H_xg.$};
\end{tikzpicture}
\caption{\footnotesize The nesting required by the proposition. Trajectories starting in $[-R,R]$ remain in $[-2R,2R]$ while a length-$h$ product is evaluated.
The cocycles $c_a$ are evaluated from a window extending a further distance $\rho$,
and all of this sits inside the periodic window.}
\label{fig:periodic}
\end{figure}

For each residue class $r\in\mathbb Z/p\mathbb Z$, define
\[
 S_r:=[-R,R]\cap(r+p\mathbb Z).
\]
It follows from \eqref{eq:local-periodicity}, that $\pi_g^x$ is a translation on $S_r$, where in general the translation parameter depends on $r$.
To appeal to Lemma \ref{lem:defect-dictionary}, we need to analyze the preimage of ${\mathbb N}_0$ under $\pi_g^x$. Because of the
uniform displacement bound  (\ref{eq:bounded-displacement}), only the integers in the interval $[-R,R]$ may have the property that the sign of $\pi_g^x(n)$ differs from that of $n$.

The preimage $(\pi_g^x)^{-1}(\mathbb N_0)\cap S_r$ is a terminal (right-most) segment of
$S_r$, Figure \ref{fig:periodic}, in particular it is determined by its cardinality.  By
\eqref{eq:bounded-displacement}, membership in
$(\pi_g^x)^{-1}(\mathbb N_0)$ agrees with membership in $\mathbb N_0$
outside $[-R,R]$.  It follows that these cardinalities---the ordered $p$-tuple of integers
\[
 \bigl(\,|(\pi_g^x)^{-1}(\mathbb N_0)\cap S_r|\,\bigr)_{r\in\mathbb Z/p\mathbb Z}
\]
determines the whole set $(\pi_g^x)^{-1}(\mathbb N_0)$, and hence the
coset $H_xg$ by Lemma~\ref{lem:defect-dictionary}.
The elements of $S_r$ are spaced by $p$, so $|S_r|\leq 2Dh/p+1$.
There are at most $2Dh/p+2$ choices for each cardinality and therefore at
most $(2Dh/p+2)^p$ cosets.
\end{proof}

Note that only the indicated finite block in $x$ is periodic in the proof above, there is no assumption that there is a periodic point in $X$.

\subsection{Two collision stages} \label{sec:two stages} Next we apply the periodic block analysis to formulate a bound on the collision rank.

\begin{proposition}\label{thm:power}
Suppose there are arbitrarily high powers of nonempty
words in the language $\mathcal L(X)$. Then every subgroup of $[[T]]$ has collision rank at
most two.
\end{proposition}

The main application in the next section will be to the commutator subgroup $[[T]]'$.

{\em Proof of Proposition \ref{thm:power}.}
Let $A\subseteq[[T]]$ be a label set.  Choose $D,\rho$ as above.
To explain the strategy, the goal is to show that the bound of Proposition \ref{prop:periodic-window} gives an inequality with the exponential number of leaves in a tree of a suitable height $h$. We will take $h=Kp$, so the coset bound in
Proposition~\ref{prop:periodic-window} becomes 
\begin{equation} \label{eq: exponent p} (2DK+2)^p, \quad 
\text{ while } 2^h=(2^K)^p.
\end{equation}Taking the logarithm of both expressions motivates the first step: consider an integer $K$ large enough so that
\begin{equation}\label{eq:log}
 2DK+2<2^K.
\end{equation}
This choice of $K$ is independent of the period $p$ defined next.
By assumption, there is a power $u^e\in\mathcal L(X)$ of a nonempty word $u$ and
$e>4DK+2\rho+1$. As explained below, $e$ is chosen to ensure the periodic interval
required by Proposition~\ref{prop:periodic-window}. Set $p:=|u|$ and $h:=Kp$.
The block $u^e$ has length
\begin{equation} \label{eq: inequality p}
 pe>4Dh+2\rho+1.
\end{equation}
Shifting an element of $X$ containing this block if necessary, we get $x\in X$
which is $p$-periodic throughout the interval $[-2Dh-\rho,2Dh+\rho]$.
Proposition~\ref{prop:periodic-window} gives
\[
 \#\{H_xg:g\in A^h\}
 \leq(2DK+2)^p<2^{Kp}=2^h<3\cdot2^{h-1}.
\]
Since $H_x$ is locally finite, Lemma~\ref{lem:coset-criterion} gives the
two collision stages
\[
 A\tarrow C\tarrow\{1\},\qquad
 C=H_x\cap A^h(A^h)^{-1}.
\]
This proves $\crk([[T]])\leq2$; the claim about subgroups follows
from subgroup monotonicity of collision rank, formally stated as Proposition \ref{prop:basic properties}(1).
\qed
\begin{remark}
Note that choosing larger powers $u^e$ may require larger period $p$. Perhaps counterintuitively, this does not affect the proof since $p$ enters the proof in two ways: (1) as the exponent of both expressions in (\ref{eq: exponent p}), so taking log makes the inequality (\ref{eq:log}) independent of $p$, and also (2) on both sides of the inequality (\ref{eq: inequality p}) which reads $pe>4DKp+2\rho+1$. Dividing it by $p$ gives \[e>4DK+\frac{2\rho+1}{p},\]
which is satisfied for any $p\geq 1$ because of the choice $e>4DK+2\rho+1$.
\end{remark}

\section{Sturmian examples}\label{sec:examples}

We first combine the estimate of Section~\ref{sec:dynamics} with the
known properties of topological full groups to state a general theorem. After the proof of this result we present a family of Sturmian groups satisfying the required periodicity condition. 

\begin{theorem}\label{thm:infinite}
Let $(X,T)$ be a minimal aperiodic subshift over a finite alphabet.
Suppose that for every $e\geq1$ there is a nonempty word $u$ with
$u^e\in\mathcal L(X)$. Then every subgroup of $[[T]]$ has 
collision rank at most two and is good. Moreover, if
\[
 [[T]]'\leq\Lambda\leq[[T]],
\]
then $\Lambda$ is infinite and amenable,
$\Lambda\notin\SG$, and $\crk(\Lambda)=2$.
In particular, the commutator subgroup $G=[[T]]'$ is an infinite finitely generated simple
amenable good group outside $\SG$. It is not finitely presented.
\end{theorem}

\begin{proof}
Proposition~\ref{thm:power} establishes an upper bound on the rank, and
Theorem~\ref{thm:main} gives the good group implication.

Next we discuss the lower bound. An infinite minimal subshift is expansive and cannot be conjugate to an
odometer, cf. \cite[Section 4.1]{FW}, and therefore by \cite[Theorem 2.4]{Matui13}, $[[T]]'$ contains the lamplighter
group $C_2\wr\Z$. It is a well-known fact that this group contains a free submonoid of rank $2$, and hence of any rank. 
In more detail, let $a$ toggle the lamp
at the origin and let $t$ be the shift in the lamplighter group. Then the elements $t,at$ freely
generate a rank $2$ free monoid, and
the three words $t^2, t(at), (at)^2$ 
generate a free submonoid of rank three. By
Lemma~\ref{lem:free-monoid}, $\crk(G)\geq2$, and by subgroup
monotonicity, $\crk(\Lambda)=2$ for every  $\Lambda$ in the statement of the theorem.

The rest of the statements follow from results in the literature, as follows. The system $(X,T)$ is Cantor minimal. Simplicity
and finite generation of $G$ follow from \cite[Theorems 4.9 and 5.4]{Matui}, and
amenability of $[[T]]$ is given by \cite[Theorem A]{JM}.
The lamplighter subgroup makes $G$ infinite and of exponential growth.
The fact that $G\notin\SG$ is proved in
\cite[Proposition 2.4]{GM}, and $G$ is not finitely presented by
\cite[Theorem 5.7]{Matui}. Subgroups of amenable groups are amenable, and if $\Lambda$ were an element of $\SG$ than so would $G$. This completes the proof of the theorem.
\end{proof}

\subsection{Sturmian groups}\label{sec:sturmian}

Given an irrational $\alpha\in(0,1)$, let
\[
 x_k=\lfloor(k+2)\alpha\rfloor-\lfloor(k+1)\alpha\rfloor,
 \qquad k\in\Z.
\]
The Sturmian subshift
of slope $\alpha$ is a classical construction \cite{MH} given by the orbit closure $X_\alpha$ of $x$, with left shift $T_\alpha$; it is minimal and
aperiodic. The symbolic sequence in $\{ 0,1\}$ associated with $\alpha$ may be read off geometrically as follows: consider the line in ${\mathbb R}^2$ of slope $\alpha/(1-\alpha)$ and record consecutive intersections with vertical and horizontal lines of the integer lattice ${\mathbb Z}^2$ as $0$, respectively $1$. The study of such sequences dates back to Johann III Bernoulli in 1772.

Consider the continued faction expansion $\alpha=[0;a_1,a_2,\ldots]$.
Following the terminology in Sturmian systems, the {\em standard words} are
\[
 s_{-1}=1,\qquad s_0=0,\qquad s_1=0^{a_1-1}1,
 \qquad s_{n+1}=s_n^{a_{n+1}}s_{n-1}\quad(n\geq1).
\]
They occur in the language $\mathcal L(X_\alpha)$, and in particular
\begin{equation}\label{eq:standard powers}
 s_n^{a_{n+1}}\in\mathcal L(X_\alpha)
 \qquad(n\geq1);
\end{equation}
see \cite[Chapter 2]{L}. 

When the continued fraction coefficients are unbounded, Equation (\ref{eq:standard powers}) gives the condition on the unbounded powers $u^e$ in the statement of Theorem\ref{thm:infinite}, giving the following consequence.

\begin{corollary}\label{cor:sturmian}
If the partial quotients $a_i$ of $\alpha=[0;a_1,a_2,\ldots]$ are unbounded, then
$G_\alpha=[[T_\alpha]]'$ satisfies all the conclusions of
Theorem \ref{thm:infinite}. In particular, $G_\alpha$ is good,
$G_\alpha\notin\SG$, and $\crk(G_\alpha)=2$.
\end{corollary}

For example, for $\alpha=[0;2,3,4,5,\ldots]$, the words $s_n^{n+2}$ occur in
the language, giving an instance satisfying Corollary \ref{cor:sturmian}.

\begin{remark}
The power condition can equivalently be stated as $\CE(X)=\infty$,
where the {\em critical exponent} is defined as
\[
 \CE(X)=\sup_{\varnothing\ne v\in\mathcal L(X)}
             \frac{|v|}{\operatorname{per}(v)}
\]
and $\operatorname{per}(v)$ is the least period of $v$.
For Sturmian words, infinite critical exponent is equivalent to unbounded
partial quotients \cite{Mignosi}. Such slopes have full
Lebesgue measure in $(0,1)$ \cite{Khinchin}.
\end{remark}

\subsection{Continuum many groups} The following lemma completes the proof of Theorem \ref{thm:main-intro}.

\begin{lemma}\label{thm:continuum}
There are $2^{\aleph_0}$ pairwise nonisomorphic finitely generated
infinite simple amenable good groups outside $\SG$ with
collision rank two.
\end{lemma}

\begin{proof}
Consider the slopes
\[
 \mathcal S=\{[0;a_1,a_2,\ldots]:a_n\in\{n+1,n+2\}\}.
\]
Distinct infinite choices give distinct irrational numbers, so
$|\mathcal S|=2^{\aleph_0}$. Every slope in $\mathcal S$ has
unbounded partial quotients and satisfies Corollary~\ref{cor:sturmian}.

All slopes in $\mathcal S$ lie in $(0,1/2)$.
Sturmian systems with distinct slopes in this interval are not
conjugate, even after replacing one of the shifts by its inverse;
see \cite[p. 186]{dC}.
Then by \cite[Theorem~5.13]{BM}, the groups
$G_\alpha$, $\alpha\in\mathcal S$, are pairwise
nonisomorphic.
\end{proof}

\section{Collision rank properties}\label{sec:free}

This section establishes some basic properties of collision rank. This includes its analysis under extensions, subgroups and quotients, characterization of $\crk\leq 1$ as supramenability, and the calculation of collision rank of free nonamenable groups. Throughout this section, it will convenient to refer to the tree $T_h$ in Definition \ref
{def:stage} as a {\em witness} of $A\tarrow B$.

Recall that a group is called {\em locally subexponential} if all its finitely
generated subgroups have subexponential growth.

\begin{proposition}\label{prop:basic properties}
Collision rank has the following properties.
\begin{enumerate}
\item If $H\leq G$, then $\crk(H)\leq\crk(G)$.
\item If $G\twoheadrightarrow Q$, then $\crk(Q)\leq\crk(G)$.
\item If $1\to N\to G\to Q\to1$ is exact, then
\[
 \crk(G)\leq\crk(N)+\crk(Q).
\]
\item If $G$ is subexponential or more generally locally subexponential, then $\crk(G)\leq1$.
\end{enumerate}

In particular, every nontrivial abelian group has rank one, and a
solvable group of derived length $d$ has rank at most $d$.

\end{proposition}

\begin{proof}
(1) For a label set $A_0$ in $H$, take a chain
\[
 A_0\tarrow A_1\tarrow\cdots\tarrow A_{\crk(G)}=\{1\}.
\]
in $G$ and intersect every
intermediate set $A_i$ with $H$. All path products then lie in $H$, so the
same trees give a chain in $H$. The intersections remain
symmetric and contain $1$.

(2) For a quotient $q:G\twoheadrightarrow Q$, consider a lift of an initial label set $A_0$ 
in $Q$ to a finite symmetric set in $G$ containing $1$, then take a chain in
$G$, and project it back to $Q$. Each labeling in $Q$ lifts to a labeling in $G$,
so the projected sets and the same trees give a chain in $Q$.

(3) For the extension claim, let $q$ denote the homomorphism $G\to Q$ and suppose $r=\crk(Q)$ and $s=\crk(N)$ are finite. Given a label set $A_0\subseteq G$, consider
a chain
\[
 q(A_0)=\overline A_0\tarrow\overline A_1\tarrow\cdots
 \tarrow\overline A_r=\{1\}
\]
in $Q$. The chain is lifted inductively, as follows. Suppose $A_0\tarrow\ldots\tarrow  A_i$ has been constructed. Starting with 
$q(A_i)\subseteq\overline A_i$ and $T_i$ witnessing
$\overline A_i\tarrow\overline A_{i+1}$, let $A_{i+1}$ consist of
all differences $p(v)\cdot p(w)^{-1}$ from $A_i$-labelings of the same tree $T_i$ and
incomparable pairs whose projected difference lies in
$\overline A_{i+1}$. This is a finite symmetric set: reversing a pair
inverts its difference, and $\overline A_{i+1}$ is symmetric. The
constant labeling by $1$ shows that $1\in A_{i+1}$. Thus
\[
 A_i\tarrow A_{i+1},\qquad
 q(A_{i+1})\subseteq\overline A_{i+1}.
\]
After $r$ steps, we get $A_r\subseteq N$. A chain for $G$ is obtained by appending an $s$-step chain in $N$.

(4) Finally, if $G$ is locally subexponential and $A\subseteq G$ is a
label set, then $|A^h|<3\cdot2^{h-1}$ for all sufficiently large
$h$. Two leaves of any $A$-labeling of $T_h$ have equal products, so
$A\tarrow\{1\}$. The assertion for solvable groups follows by
induction using the extension inequality.
\end{proof}

Recall that a group is called  \emph{supramenable} \cite{R} if none of its nonempty subsets is
paradoxical: no subset contains two disjoint copies of itself
obtained by finite partitions and left translations. For example, a group containing a free monoid $S$ of rank $2$ is not supramenable: denoting free monoid generators by $x$ and $y$, $S$ contains two disjoint translated copies $xS, yS$ of itself. Subexponential groups are known to be supramenable \cite{R} while the reverse implication is an open problem.

\begin{proposition}\label{lem:supramenability}
A group $G$ is supramenable if and only if $\crk(G)\leq1$.
\end{proposition}

\begin{proof}
We use the characterization of supramenability in
\cite[Definition~3.1 and Proposition~3.4]{KMR}.
Let $F_2=\langle a,b\rangle$, and consider its Cayley graph, the
four-valent tree with edges labeled by $a^{\pm 1}, b^{\pm 1}$.
The characterization says 
$G$ is not supramenable exactly when there is an injective map $f:F_2\to G$ and a finite set $S\subseteq G$ such that $f(v)\cdot f(w)^{-1}\in S$ whenever $v,w$ are adjacent in the Cayley tree of $F_2$.
This is the Lipschitz condition of \cite{KMR};
$f$ is not required to be a group homomorphism.

The characterization in \cite{KMR} is close to our definition of $\crk(G)\leq 1$, a minor difference is in the degree of the trees ($4$ versus $3$) and additionally, in the collision rank definition only pairs of incomparable vertices are considered. The proof below mainly addresses this discrepancy. 

Suppose first that such a map $F_2\longrightarrow G$ exists. Restrict it to an infinite
rooted trivalent subtree, with root $o$. Multiplying all values of
$f$ on the right by $f(o)^{-1}$, we may assume that $f(o)=1$;
this preserves both injectivity and the edge label membership in $S$.
Choose a (symmetric) label set $A$ containing $S$, and label the
edge from a vertex $w$ to its child $v$ by $f(v)\cdot f(w)^{-1}$.
The product along the path from the root to $v$ equals 
$p(v)=f(v)$. Since $f$ is injective, every finite truncation gives
an $A$-labeling without a collision to $\{1\}$.
Thus $A\not\tarrow\{1\}$ and $\crk(G)>1$.

Conversely, suppose that $A\not\tarrow\{1\}$ for some label set $A$.
Let $T$ be the (abstract) infinite rooted trivalent tree obtained as the union
of the trees $T_h$. Since $A$ is finite, the space $A^{E(T)}$ of
all edge labelings of $T$ is compact in the product topology.
For each $h$, the labelings whose restriction to $T_h$ has no
collision to $\{1\}$ form a nonempty closed subset of this space.
These subsets are nested, so their intersection is nonempty.
We therefore obtain an $A$-labeling of $T$ for which
\[
p(v)=p(w)\quad\Longrightarrow\quad
 v,w\text{ are comparable}.
\]
(Alternatively, one can appeal to K\"onig's lemma to get this conclusion.) 
Thus all vertices carrying a given group element lie along a
single geodesic path starting at the root. We claim that each group
element occurs only finitely many times.
Suppose that $g$ occurs at vertices $v_1,\ldots,v_m$, listed in
increasing height. For each $i<m$, choose a child of $v_i$ that
does not lie on the path to $v_{i+1}$.
These $m-1$ children are pairwise incomparable, so the product group elements associated with them
are distinct. On the other hand, all their products belong to
the finite set $Ag$, therefore $m-1\leq |A|$.
In particular, every group element $g$ in the image of the map $p$ has a unique
last occurrence, denoted $v(g)$, furthest from the root.

We now make a new assignment of group elements to the vertices of an infinite
rooted binary tree, starting with $1$ at its root.
Whenever a vertex has been assigned the value $g$, consider
$v(g)$ in the original trivalent tree. (So for the first step of the construction, it is $v(1)$.) Choose two children
$u_0,u_1$ of $v(g)$, and assign
\begin{equation} \label{eq:g}
 g_0=p(u_0),\qquad g_1=p(u_1)
\end{equation}
to the two children of the binary-tree vertex.

We claim that this defines an injective map from the infinite binary tree to $G$, in other words that all group elements assigned to vertices in this construction are
distinct. For $g_j$ in (\ref{eq:g}), the last occurrence
$v(g_j)$ lies in the descendant subtree rooted at $u_j$, so
the two constructions continue in disjoint subtrees.
Values assigned in different branches cannot agree, because
they occur at incomparable vertices of $T$, and a value assigned
to an ancestor cannot reappear either.
The resulting map from the binary tree to $G$ is therefore
injective. Moreover, each parent-child pair satisfies
\[
g_jg^{-1}=p(u_j)p(v(g))^{-1}\in A.
\]
To obtain a map from $F_2$, a further modification is needed. Every vertex of the constructed binary tree has four grandchildren. We form a new, $4$-valent, tree using selected vertices at even depths: join its root to all four grandchildren.
At each subsequent selected vertex, join it to three of its four grandchildren.

Restricting to these selected vertices therefore gives an injective map
$ F_2\longrightarrow G$,
whose edge labels belong to the finite set $A^2$. It is Lipschitz in the sense of \cite{KMR}, so $G$ is not supramenable.
\end{proof}

We conclude with an example of infinite collision rank.
Call a subset of a group a \emph{free family of rank $m$} if its elements freely
generate a free subgroup of rank $m$.

\begin{lemma}\label{lem:descent}
Suppose $A\tarrow B$. If $A$ contains a free family of rank $3m$,
then $B$ contains a free family of rank $m$.
\end{lemma}

\begin{proof}
Consider a partition of the free family into triples $C_1,\ldots,C_m$. For each
$i$, label a witness tree using only $C_i$, with distinct labels on
the edges adjacent to each vertex. As in Lemma~\ref{lem:free-monoid},
distinct vertices have distinct products. Denoting by $\langle C_i\rangle $ the subgroup generated by $C_i$, the collision gives
\[
 g_i\in B\cap\bigl(\langle C_i\rangle\setminus\{1\}\bigr).
\]
The groups $\langle C_i\rangle$ are free factors in their free
product, and each $g_i$ has infinite order. This implies that $g_1,\ldots,g_m$ freely generate a free
group of rank $m$.
\end{proof}

\begin{proposition}\label{thm:free}
Every group containing a nonabelian free subgroup has infinite
collision rank.
\end{proposition}

\begin{proof}
Suppose $\crk(G)=r<\infty$. A nonabelian free subgroup contains a
free family of any finite rank, in particular of rank $3^r$. Complete this family to a symmetric label set $A_0$,
and consider a chain
\[
 A_0\tarrow A_1\tarrow\cdots\tarrow A_r=\{1\}.
\]
An inductive application of Lemma \ref{lem:descent} implies that $A_i$ contains
a free family of rank $3^{r-i}$. In particular, $\{1\}$ contains a
free family of rank one, a contradiction.
\end{proof}

\end{document}